\documentclass[reqno]{amsart}

\usepackage{amsmath}
\usepackage{amsfonts,amssymb}
\usepackage[dvips]{graphicx}
\usepackage{epsfig}
\usepackage{color}
\usepackage{fancybox,ascmac}
\newtheorem{theorem}{Theorem}[section]
\newtheorem{remark}{Remark}[section]

\newtheorem{definition}{Definition}[section]

\newtheorem{proposition}[theorem]{Proposition}

\numberwithin{equation}{section}
\begin{document}

\title[the compressible Euler equations in the isentropic nozzle flows ]
	{Global existence of a classical solution for the isentropic nozzle flow}

\author{Shih-Wei Chou}
\address{Department of Finance Engineering and Actuarial Mathematics, Soochow University, Taipei 10048, Taiwan.}
\email{swchou@scu.edu.tw}\thanks{
	S-W. Chou's research is partially supported by the National Science and Technology Council, Taiwan. under the grants MOST 111-2115-M-031 -004 -MY2.
}

\author{Bo-Chih Huang}
\address{Department of Mathematics, National Central University, Chung-Li 32001, Taiwan.}
\email{huangbz@math.ncu.edu.tw}\thanks{
	B-C. Huang's research is partially supported by the National Science and Technology Council, Taiwan. under the grants MOST 111-2115-M-194 -003 -MY2.
}

\author{Yun-guang Lu}
\address{\textcolor{black}{School of Mathematics, Hangzhou Normal University, Hangzhou
		311121, China.}}
\email{ylu2005@ustc.edu.cn}
\thanks{\textcolor{black}{Y.G. Lu 's research is partially supported by the National Natural Science Foundation of China with the Grant No. 12071106.}}

\author{Naoki Tsuge}
\address{Department of Mathematics Education, 
Faculty of Education, Gifu University, 1-1 Yanagido, Gifu
Gifu 501-1193 Japan.}
\email{tsuge.naoki.c9@f.gifu-u.ac.jp}
\thanks{
N. Tsuge's research is partially supported by Grant-in-Aid for Scientific 
Research (C) 17K05315, Japan.
}
\keywords{The Compressible Euler Equation, the nozzle flow, classical solutions, time global solution, the invariant regions, the comparison theorem.}
\subjclass{Primary 
35F31
35L04
35L65, 
35Q31, 
76N10,
76N15; 
Secondary
35A01, 
35B35,   
35B50, 
35B51
35L60,   
}
\date{}

\maketitle
\begin{abstract}
Our goal in this paper is to prove the global existence of a classical solution for the isentropic nozzle flow. Regarding this problem, there exist some global existence theorems of weak solutions. However, that of classical solutions does not have much attention until now.
When we consider the present problem, the main difficulty is to obtain the uniform bound of 
solutions and their derivatives. To solve this, 
we introduce an invariant region depending on the space variable and a functional satisfying the Riccati equation along the characteristic lines.
\end{abstract}

\section{Introduction}
The present paper is concerned with isentropic gas flow in a nozzle on a half space ${\bf R}_+=\left\{x\in{\bf R};x\geq0\right\}$.
This motion is governed by the following compressible Euler equations:
\begin{equation}\left\{\begin{array}{ll}
\displaystyle{\rho_t+m_x=-a(x)m,}\\
\displaystyle{m_t+\left(\frac{m^2}{\rho}+p(\rho)\right)_x
=-a(x)\frac{m^2}{\rho},\quad x\geq0},
\end{array}\right.
\label{eqn:nozzle}
\end{equation}
where $\rho$, $m$ and $p$ are the density, the momentum and the 
pressure of the gas, respectively. If $\rho>0$, 
$v=m/\rho$ represents the velocity of the gas. For a barotropic gas, 
$p(\rho)=\rho^\gamma/\gamma$, where $\gamma\in(1,5/3]$ is the 
adiabatic exponent for usual gases. The given function $a(x)$ is 
represented by 
\begin{align*}
a(x)=A'(x)/A(x),
\end{align*}
where $A\in C^2({\bf R}_+)$ is a slowly variable cross section area at $x$ in the nozzle satisfying $A(x)>0$.

We introduce the Riemann invariants $w,z$, which play important roles
in this paper, as
\begin{definition}
\begin{align*}
w=\frac{m}{\rho}+\frac{\rho^{\theta}}{\theta}=v+\frac{\rho^{\theta}}{\theta},
\quad{z}=\frac{m}{\rho}-\frac{\rho^{\theta}}{\theta}
=v-\frac{\rho^{\theta}}{\theta}\quad\left(\theta=\dfrac{\gamma-1}{2}\right).
\end{align*}
\end{definition}
These Riemann invariants satisfy the following.
\begin{remark}\label{rem:Riemann-invariant}
\normalfont
\begin{align}
&v=\frac{w+z}2,
\;\rho=\left(\frac{\theta(w-z)}2\right)^{1/\theta},\;m=\rho v.
\label{eqn:relation-Riemann}
\end{align}
\end{remark}
If \eqref{eqn:nozzle} has a smooth solution, we can diagonalize \eqref{eqn:nozzle} into 
\begin{align}
\begin{cases}
z_t+\lambda_1z_x=\dfrac{\gamma-1}{8}a(x)(w^2-z^2),\\
w_t+\lambda_2w_x=-\dfrac{\gamma-1}{8}a(x)(w^2-z^2),
\end{cases}
\label{eqn:diagonalization}
\end{align}
where $\lambda_1$ and $\lambda_2$ are the characteristic speeds defined as follows 
\begin{align}
\lambda_1=v-\rho^{\theta},\quad\lambda_2=v+\rho^{\theta}.
\label{char}
\end{align}

We consider three initial boundary problems (P1) with \eqref{eqn:diagonalization}, \eqref{eqn:I.C.}, \eqref{eqn:B.C.1}, (P2) with \eqref{eqn:diagonalization}, \eqref{eqn:I.C.}, \eqref{eqn:B.C.2} and (P3) with \eqref{eqn:diagonalization}, \eqref{eqn:I.C.},  where  
\begin{align}  
&(\rho,m)|_{t=0}=(\rho_0(x),m_0(x)),\quad x\geq0,
\label{eqn:I.C.}\\
&v|_{x=0}=0,
\label{eqn:B.C.1}\\
&(z,w)|_{x=0}=(z_B(t),w_B(t)),\quad t\geq0.
\label{eqn:B.C.2}
\end{align}

The equation \eqref{eqn:nozzle} can be written in the following form 
\begin{align}
u_t+F(u)_x=G(x,u),\quad{x}\in{\bf R}_+,
\end{align}
by using  $u=\begin{pmatrix}
\rho\\m
\end{pmatrix}$, $\displaystyle F(u)=\begin{pmatrix}
m\\\dfrac{m^2}{\rho}+p(\rho)
\end{pmatrix}$ and 
$G(x,u)=\begin{pmatrix}
	-a(x)m\\\displaystyle -a(x)\dfrac{m^2}{\rho}
	\end{pmatrix}$.

From the viewpoint of application, let us review \eqref{eqn:nozzle}. In engineering, nozzles
are useful in various areas. One of the most famous nozzles is the Laval nozzle. It is
a tube that is pinched in the middle, making a hourglass-shape. The Laval nozzle
accelerates a subsonic to a supersonic flow. Because of this property, the nozzle is
widely utilized in some type of turbine, which is an essential part of the modern
rocket engine or the jet engine.

From the mathematical point of view, \eqref{eqn:nozzle} is one of typical equations in the inhomogeneous conservation law and is categorized as the quasi-linear hyperbolic equation. Even if initial data are smooth, such a equation 
has discontinuities in general.  
The pioneer work in this direction is Liu \cite{L1}.
In \cite{L1}, Liu proved the existence of global weak solutions coupled with steady states, by the Glimm scheme, provided that the initial data have small total variation and
are away from the sonic state. {\color{black}Since then, the existence of the global weak solutions for various models have been studied, see \cite{CSW,DH,HCHY,MB} and therein.} Recently, the existence theorems that include the transonic state have been obtained. {\color{black}The transonic stationary solutions has studied in \cite{HYH,HHL,HHL2}.}
The author {\color{black}generalize the invariant region theory in \cite{CS} to} proved the global existence of {\color{black}bounded} weak solutions for the Laval nozzle \cite{T1,T4} and the general nozzle \cite{T2,T3} by the compensated compactness. {\color{black}In \cite{HCHY}, the authors proved the global existence of weak solutions with bounded variation for flows in general nozzle with extra outer forces.}

On the other hand, the existence of classical solutions does not receive much attention until now. {\color{black}Recently, the authors in \cite{CHL2} consider the initial-boundary value problem for \eqref{eqn:nozzle} with large $C^0$ initial-boundary data. They establish the global existence and asymptotic behavior of classical solutions for supersonic flows through the nozzle. The work is based on the local existence, the maximum principle, and the uniform a priori estimates obtained by the generalized Lax transformations. In \cite{CHL}, the authors extended the result in \cite{CHL2} to the ducts depending on both space and time. However, under the a priori estimate in \cite{CHL,CHL2}, the authors constrain themselves to discuss the gas near vacuum on the expanding nozzle. In our work, by choosing appropriate invariant region, we can extend the result for the existence of classical solutions of \eqref{eqn:nozzle} to general nozzle and when the initial-boundary data is far away from vacuum.} \textcolor{black}{In addition, we can treat with the wider range of data.}

To state our main theorem, we consider a function 
\begin{align*}
	f(r)=\dfrac{2}{\gamma-1}\dfrac{\gamma+1+(3-\gamma)r}{|r^2-1|}.
\end{align*}
This function has a minimum on $[-1,1]$. We call the minimum $l$. Then, we denote a solution of $f(r)=l$ on $r<-1$ by $-\sigma_1\;(\sigma_1>1)$; we denote a solution of $f(r)=l$ on $r>1$ by $\sigma_2\;(\sigma_2>1)$. Then we notice 
that 
\begin{align}
	f(r)\geq l\text{ on }[-\sigma_1,\sigma_2].
	\label{eqn:f(r)>l}
\end{align}

We introduce the following conditions of $a(x)=A_{x}/A(x)$:
\vspace*{1ex}

(H1) There exist positive constants $k_1,k_2,\alpha$ and $M$ such that $a(x)\in C^1_b({\bf R}_+)$  satisfy
\begin{align}
{\left(a(x)\right)^2}\leq k_1\left(1+Mx\right)^{-2-\alpha},\; 	\left|a'(x)\right|\leq k_2\left(1+Mx\right)^{-2-\alpha}\text{ for any $x\in{\bf R}_+$}.
\label{eqn:H1}
\end{align}

\vspace*{1ex}

Next, to construct invariant regions, we prepare 
positive constants $L_1,L_2,U_1,U_2$. We independently introduce the following three conditions of $a(x)$ and $L_1,L_2,U_1,U_2$:\vspace*{1ex}

	\begin{itemize}
	\item[(H2)]There exists a function $\bar{a}\in C^1_b({\bf R}_+)\cap L^1({\bf R}_+)$ and positive constants 
	$L_1,L_2,U_1,U_2$ such that 
	\begin{align}
		|a(x)|< l\bar{a}(x)
		\label{eqn:bar a}
	\end{align}
	\begin{align}
		&U_1e^{2\int^{\infty}_0\bar{a}(x)dx}\leq L_1,\quad
		L_2\leq U_2,\label{eqn:invariant1}\end{align}
	\begin{align}
		&\dfrac{3-\gamma}{\gamma+1}<\dfrac{L_2}{L_1},\quad
		\dfrac{U_2}{U_1}<\dfrac{\gamma+1}{3-\gamma},
		\label{eqn:invariant2}\end{align}
	\begin{align}
			\dfrac{L_1}{L_2}\leq\sigma_1,\quad\dfrac{U_2}{U_1}\leq\sigma_1,
		\label{eqn:invariant3}
	\end{align}
\begin{align}
	U_2\geq L_1,\quad{U_1}\geq L_2,
	\label{eqn:invariant7}
\end{align}

	\item[(H3)]There exists a function $\bar{a}\in C^1_b({\bf R}_+)\cap L^1({\bf R}_+)$ and positive constants 
	$L_1,L_2,U_1,U_2$ such that \eqref{eqn:bar a}, 
	\begin{align}
		L_1\leq U_1,\;L_2\leq U_2,
		\label{eqn:invariant4}
	\end{align}
	\begin{align}
		\begin{split}
			L_2-{U_1}e^{2\int^{\infty}_0\bar{a}(x)dx}>0,
		\end{split}
		\label{eqn:invariant5}
	\end{align}
	\begin{align}
		\begin{split}
			\dfrac{U_2}{L_1}e^{2\int^{\infty}_0\bar{a}(x)dx}\leq\sigma_2
		\end{split}.
		\label{eqn:invariant6}
	\end{align}

	\item[(H4)]There exists a function $\bar{a}\in C^1_b({\bf R}_+)\cap L^1({\bf R}_+)$ and positive constants 
$L_1,L_2,U_1,U_2$ such that \eqref{eqn:bar a}, 
\begin{align}
	L_1\geq U_1e^{2\int^{\infty}_0\bar{a}(x)dx},\;L_2\geq U_2e^{2\int^{\infty}_0\bar{a}(x)dx},
	\label{eqn:invariant8}
\end{align}
\begin{align}
	\begin{split}
		U_1>L_2,
	\end{split}
	\label{eqn:invariant9}
\end{align}
\begin{align}
	\begin{split}
		\dfrac{L_1}{U_2e^{2\int^{\infty}_0\bar{a}(x)dx}}\leq\sigma_2
	\end{split}.
	\label{eqn:invariant10}
\end{align}

\end{itemize}
\vspace*{1ex}

\begin{remark}\label{rem:M}
Let $L_1,L_2,U_1,U_2, M$ be constants satisfying (H1) and (H2). For any $\tilde{M}$ such that $\tilde{M}\geq M$,  
replacing $M$ in  \eqref{eqn:H1} with $\tilde{M}$, constants $L_1,L_2,U_1,U_2, \tilde{M}$ also 
satisfy (H1) and (H2). The similar fact holds for (H1) and (H3) (resp. (H1) and (H4)).
\end{remark}

For $L_1,L_2,U_1,U_2$ and $\bar{a}$ satisfying \eqref{eqn:bar a}--\eqref{eqn:invariant3}, we denote a set depending on $x$ by
\begin{align}
	\begin{split}
		\Delta_{m,x}=&\left\{(z,w);-L_1e^{-\int^x_0\bar{a}(y)dy}\leq z\leq -U_1e^{\int^x_0\bar{a}(y)dy},\right.\\
		&\left.L_2e^{-\int^x_0\bar{a}(y)dy}\leq w\leq U_2e^{\int^x_0\bar{a}(y)dy},w\geq z\right\};
	\end{split}
	\label{eqn:Delta1}
\end{align}

\begin{figure}[htbp]
	\begin{center}
		\hspace{-2ex}
		\includegraphics[scale=0.4]{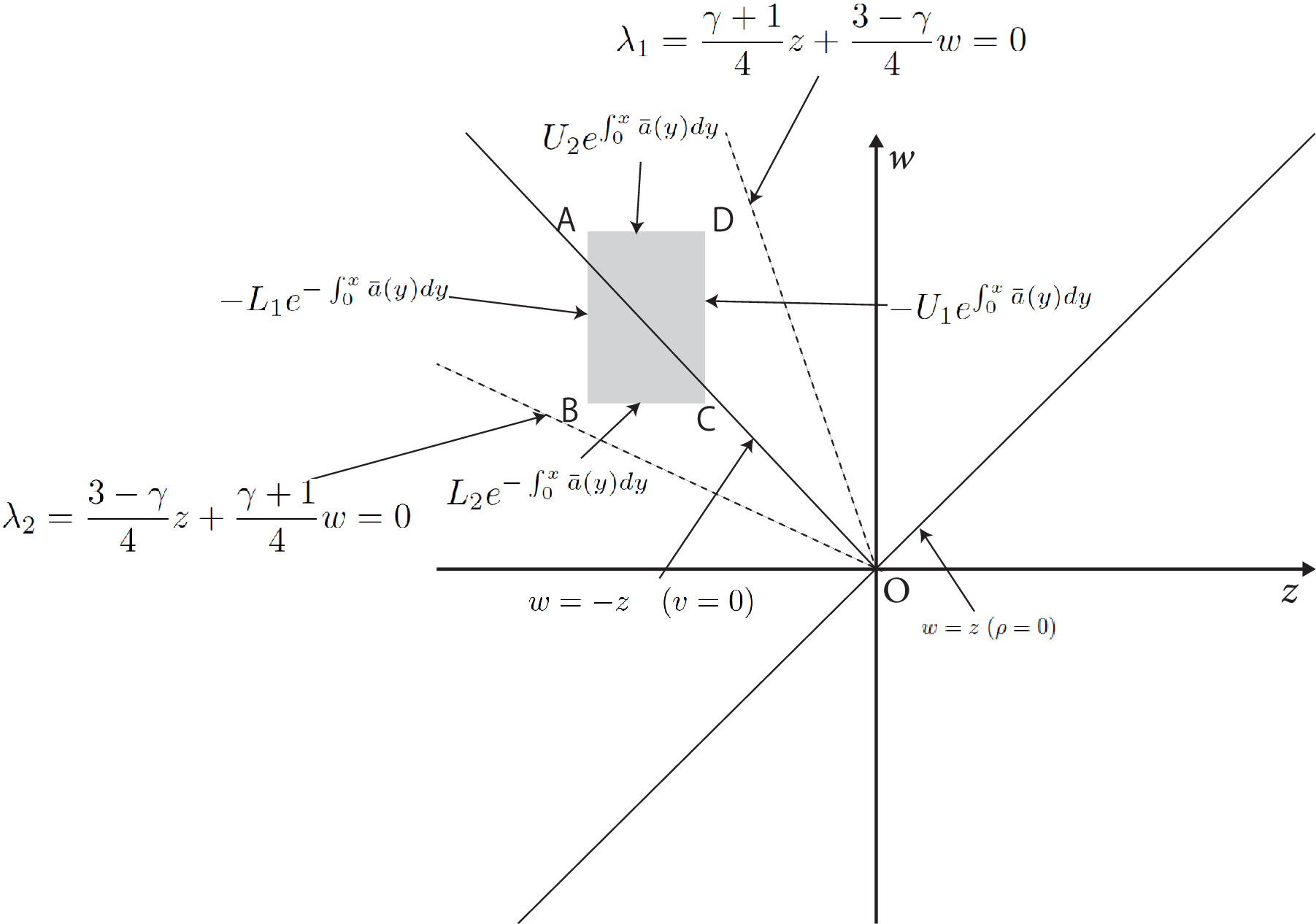}
	\end{center}
	\caption{$\Delta_{m,x}$ in $(z,w)$-plane}
	\label{Fig:1}
\end{figure}

\begin{remark}\label{rem:1}
	We observe the following properties of $\Delta_{m,x}$ (see Figure \ref{Fig:1}).
	\begin{enumerate}
		\item From \eqref{eqn:invariant1}, we find $\Delta_{m,x}\ne\emptyset$.
		\item From \eqref{eqn:invariant2}, Rectangle ABCD does not intersect $\lambda_1=\lambda_2=0$. Solutions in $\Delta_{m,x}$ satisfy $\lambda_1<0$ and $\lambda_2>0$. This means that they are the subsonic flow.
		\item AB and CD are away from $z=0$ . 
		\item BC and DA are away from $w=0$. 
		\item $L_1,\;U_1,\;L_2,\;U_2$ are very close.
		\item Rectangle ABCD is away from the vacuum $w=z$. 
		\item From \eqref{eqn:invariant7}, we find that A lies in $\{w\geq -z\}$ and C lies in $\{w\leq -z\}$.
	\end{enumerate}
\end{remark}

For $L_1,L_2,U_1,U_2$ and $\bar{a}$ satisfying \eqref{eqn:bar a}, \eqref{eqn:invariant4}--\eqref{eqn:invariant6}, we then denote a set depending on $x$ by
\begin{align}
	\begin{split}
		\Delta_{r,x}=&\left\{(z,w);L_1e^{-\int^x_0\bar{a}(y)dy}\leq z\leq \textcolor{black}{U_1}e^{\int^x_0\bar{a}(y)dy},\right.\\
		&\left.{L_2e^{-\int^x_0\bar{a}(y)dy}\leq w\leq U_2e^{\int^x_0\bar{a}(y)dy},w\geq z}\right\};
	\end{split}
	\label{eqn:Delta2}
\end{align}
\begin{figure}[htbp]
	\begin{center}
		\hspace{-2ex}
		\includegraphics[scale=0.4]{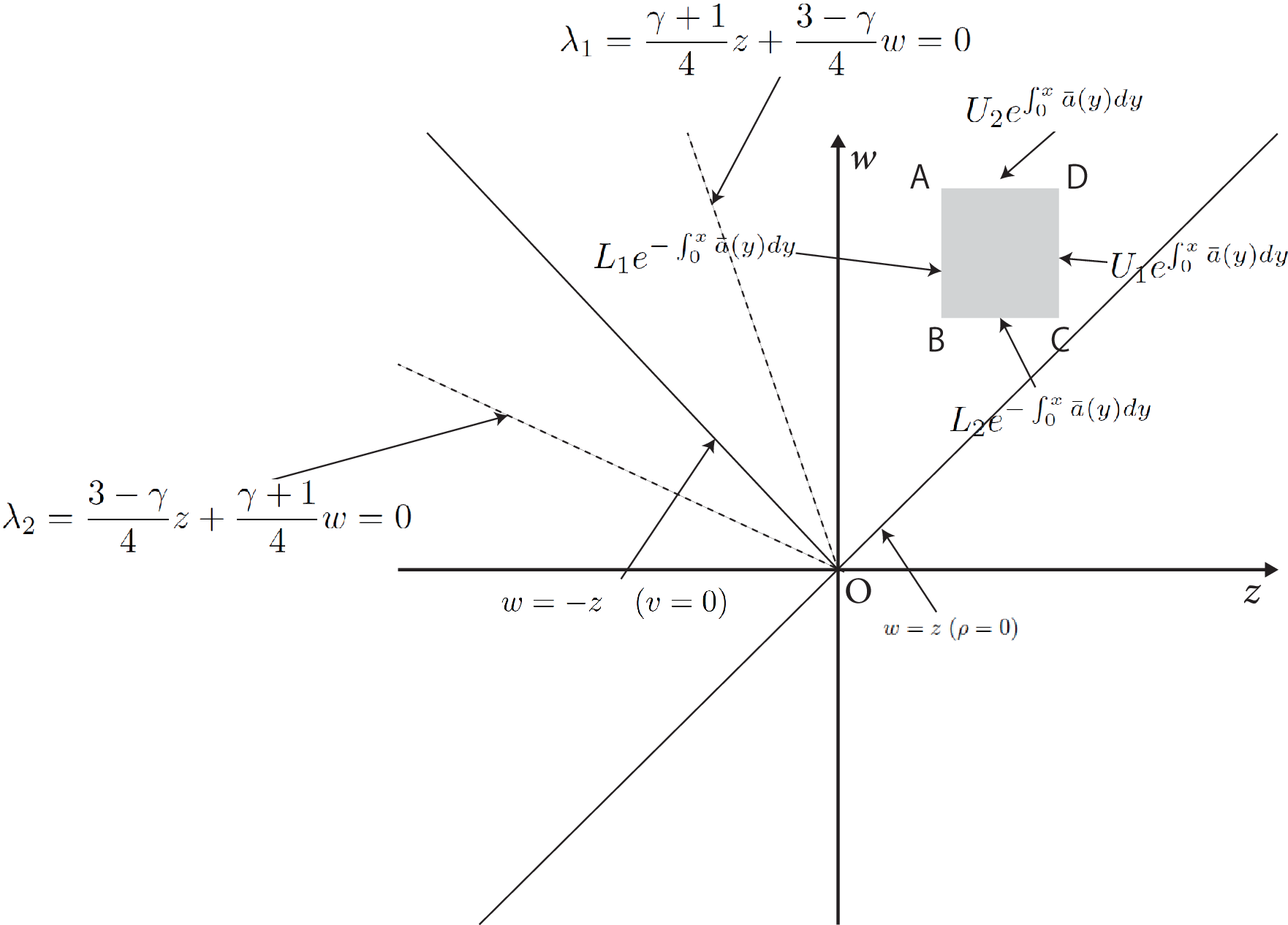}
	\end{center}
	\caption{$\Delta_{r,x}$ in $(z,w)$-plane}
	\label{Fig:2}
\end{figure} 
\begin{remark}\label{rem:2}
	We observe the following properties of $\Delta_{r,x}$ (see Figure \ref{Fig:2}).
	\begin{enumerate}
		\item From \eqref{eqn:invariant4}, we find $\Delta_{r,x}\ne\emptyset$.
		\item AB and CD are away from $z=0$. 
		\item BC and DA are away from $w=0$. 
		\item Rectangle ABCD does not intersect $\lambda_1=\lambda_2=0$. Solutions in $\Delta_{r,x}$ satisfy $\lambda_1>0$ and $\lambda_2>0$. This means that they 
		are the supersonic flow.
		\item $L_1,\;U_1,\;L_2,\;U_2$ are very close.
		\item From \eqref{eqn:invariant5}, Rectangle ABCD is away from the vacuum $w=z$. 
	\end{enumerate}
\end{remark}

For $L_1,L_2,U_1,U_2$ and $\bar{a}$ satisfying \eqref{eqn:bar a}, \eqref{eqn:invariant8}--\eqref{eqn:invariant10}, we then denote a set depending on $x$ by
\begin{align}
	\begin{split}
		\Delta_{l,x}=&\left\{(z,w);-L_1e^{-\int^x_0\bar{a}(y)dy}\leq z\leq -\textcolor{black}{U_1}e^{\int^x_0\bar{a}(y)dy},\right.\\
		&\left.-{L_2e^{-\int^x_0\bar{a}(y)dy}\leq w\leq -U_2e^{\int^x_0\bar{a}(y)dy},w\geq z}\right\}.
	\end{split}
	\label{eqn:Delta3}
\end{align}
\begin{figure}[htbp]
	\begin{center}
		\hspace{-2ex}
		\includegraphics[scale=0.4]{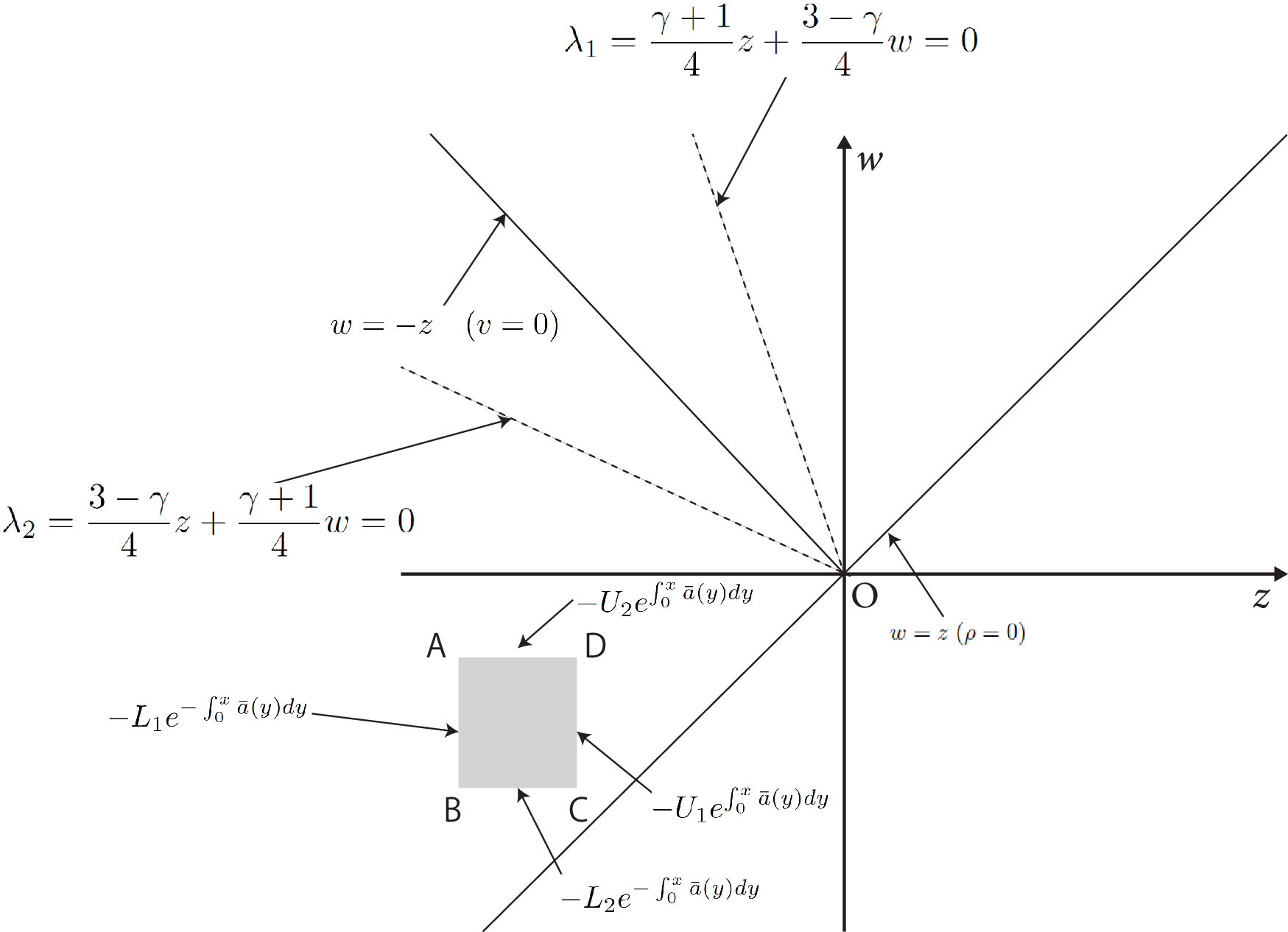}
	\end{center}
	\caption{$\Delta_{l,x}$ in $(z,w)$-plane}
	\label{Fig:3}
\end{figure}

Next, to deduce the uniformly bounded estimate of derivatives of $z$ and $w$, we employ the following, which is introduce in \cite[Section 2.2]{CHL}:

If $\beta\ne-1$,
\begin{align}
	\begin{split}
		&\Phi(x,t;z,w)=(w-z)^{\beta}z_x+\dfrac{a(x)z}{2\beta}(w-z)^{\beta}+\dfrac{a(x)}{2(\beta+1)}(w-z)^{\beta+1},\\
		&\Psi(x,t;z,w)=(w-z)^{\beta}w_x+\dfrac{a(x)w}{2\beta}(w-z)^{\beta}-\dfrac{a(x)}{2(\beta+1)}(w-z)^{\beta+1},
	\end{split}
\label{eqn:Phi-Psi1}
\end{align}\begin{align*}
	\begin{split}
		\Phi_B(x,t;z,w)=&-(w-z)^{\beta}
\dfrac1{\lambda_1}\left\{z_t-\dfrac{\gamma-1}{8}a(x)(w^2-z^2)
\right\}+\dfrac{a(x)z}{2\beta}(w-z)^{\beta}\\
&+\dfrac{a(x)}{2(\beta+1)}(w-z)^{\beta+1},\\
		\Psi_B(x,t;z,w)=&-(w-z)^{\beta}\dfrac1{\lambda_2}\left\{w_t+\dfrac{\gamma-1}{8}a(x)(w^2-z^2)
\right\}+\dfrac{a(x)w}{2\beta}(w-z)^{\beta}\\&-\dfrac{a(x)}{2(\beta+1)}(w-z)^{\beta+1};
	\end{split}
\end{align*}
\textcolor{black}{
	if $\beta=-1$,
	\begin{align}
		\begin{split}
			&\Phi(x,t;z,w)=\dfrac{z_x}{w-z}-\dfrac{a(x)z}{2(w-z)}
			+\dfrac{a(x)}{2}\log(w-z),\\
			&\Psi(x,t;z,w)=\dfrac{w_x}{w-z}-\dfrac{a(x)w}{2(w-z)}
			-\dfrac{a(x)}{2}\log(w-z),
		\end{split}
		\label{eqn:Phi-Psi2}
	\end{align}
	\begin{align*}
		\begin{split}
			\Phi_B(x,t;z,w)=&-\dfrac{1}{w-z}
			\dfrac1{\lambda_1}\left\{z_t-\dfrac{\gamma-1}{8}a(x)(w^2-z^2)
			\right\}-\dfrac{a(x)z}{2(w-z)}\\
			&+\dfrac{a(x)}{2}\log(w-z),\\
			\Psi_B(x,t;z,w)=&-\dfrac{1}{w-z}\dfrac1{\lambda_2}\left\{w_t+\dfrac{\gamma-1}{8}a(x)(w^2-z^2)
			\right\}-\dfrac{a(x)w}{2(w-z)}\\&-\dfrac{a(x)}{2}\log(w-z),
		\end{split}
\end{align*}}
where \begin{align}
	\beta=\dfrac{\gamma-3}{2(\gamma-1)}.
	\label{eqn:beta}
\end{align}

Then, our main theorems are as follows.
\begin{theorem}
We assume that 
\begin{enumerate}
\item $a(x)$ and $L_1,L_2,U_1,U_2$ satisfy (H1) and (H2),
\item $(z,w)\in C^1({\bf R}_+)$ satisfies $(z_0(x),w_0(x))\in\Delta_{m,x}$ and \begin{align}
\begin{split}
-\delta_1\left(1+Mx\right)^{-1-\alpha}&\leq\Phi(x,0;z_0,w_0)\leq\delta_2,\\
\delta_1\left(1+Mx\right)^{-1-\alpha}&\leq\Psi(x,0;z_0,w_0)\leq\delta_2,
\end{split}
\label{eqn:condition-initial}
\end{align}for some positive constants $\delta_1\leq \delta_2$,
\item the compatibility conditions: $w_0(0)+z_0(0)=0,\;w'_0(0)-z'_0(0)=0$.
\end{enumerate}
Then, if we choose $M$ in \eqref{eqn:H1} large enough, 
the initial boundary problem (P1) has a time global classical solution.
\end{theorem}

\begin{theorem}
We assume that 
\begin{enumerate}
\item $a(x)$ and $L_1,L_2,U_1,U_2$ satisfy (H1) and (H3), 
\item $(z,w)\in C^1({\bf R}_+)$ satisfies $(z_0(x),w_0(x))\in\Delta_{r,x},\;(z_B(t),w_B(t))\in\Delta_{r,0}$ and \begin{align}&
\begin{split}
&\delta_1\left(1+Mx\right)^{-1-\alpha}\leq\Phi(x,0;z_0,w_0)\leq\delta_2,\\
&\delta_1\left(1+Mx\right)^{-1-\alpha}\leq\Psi(x,0;z_0,w_0)\leq\delta_2,
\end{split}
\label{eqn:condition-initia2}
\end{align}
\begin{align}
\begin{split}
&\delta_1\leq\Phi_B(0,t;z_B,w_B)\leq\delta_2,\\
&\delta_1\leq\Psi_B(0,t;z_B,w_B)\leq\delta_2,
\end{split}
\label{eqn:condition-boundary}
\end{align}for some positive constants $\delta_1\leq \delta_2$, 
\item the compatibility conditions: $z_0(0)=z_B(0),\;w_0(0)=w_B(0)$ and 
\begin{align*}
\begin{cases}
z'_B(0)+\lambda_1(u_0(0))z'_0(0)=\dfrac{\gamma-1}{8}a(0)
\left[\left\{w_0(0)\right\}^2-\left\{z_0(0)\right\}^2\right],\\
w'_B(0)+\lambda_2(u_0(0))w'_0(0)=-\dfrac{\gamma-1}{8}a(0)\left[\left\{w_0(0)\right\}^2-\left\{z_0(0)\right\}^2\right].
\end{cases}
\end{align*}
\end{enumerate}
Then,  if we choose $M$ in \eqref{eqn:H1} large enough, 
the initial boundary problem (P2) has a time global classical solution.
\end{theorem}

\begin{theorem}
	We assume that 
	\begin{enumerate}
		\item $a(x)$ and $L_1,L_2,U_1,U_2$ satisfy (H1) and (H4), 
		\item $(z,w)\in C^1({\bf R}_+)$ satisfies $(z_0(x),w_0(x))\in\Delta_{l,x}$ and \begin{align}&
			\begin{split}
				&-\delta_1\left(1+Mx\right)^{-1-\alpha}\leq\Phi(x,0;z_0,w_0)\leq\delta_2,\\
				&-\delta_1\left(1+Mx\right)^{-1-\alpha}\leq\Psi(x,0;z_0,w_0)\leq\delta_2,
			\end{split}
			\label{eqn:condition-initia3}
		\end{align}for some positive constants $\delta_1\leq \delta_2$.
	\end{enumerate}
	Then, if we choose $M$ in \eqref{eqn:H1} large enough, 
	the initial boundary problem (P3) has a time global classical solution.
\end{theorem}

To prove the above theorems, we prepare the following theorem. 
\begin{theorem}\label{thm:local} (\cite[Section 3.8]{B})
	Suppose that $z_0,w_0\in C^1_b({\bf R}_+)$, that $z_B,w_B\in C^1_b({\bf R}_+)$, $a\in C^1_b({\bf R}_+)$ and that the  compatibility conditions hold. Then there exists a positive constant $T$ such that three problems (P1), (P2) and (P3) have a unique bounded $C^1$ solution on ${\bf R}_+\times [0,T]$ respectively, where $T$ depends only on $C^1$ 
	norms of $z_0, w_0, z_B, w_B, a$.
\end{theorem}

To extend the above time local solution to a time global one, we must obtain the uniformly bounded estimate of 
$z,w$ and their derivatives. In Section 2, we develop an invariant region depending on the space variable, which is introduced in \cite{T2} and \cite{T3}. The invariant region will 
yields the uniformly bounded estimate of $z,w$ and the lower bound of $\rho$. In Section 3, 
we calculate \eqref{eqn:Phi-Psi1} and \eqref{eqn:Phi-Psi2} to deduce the uniform bound of $z_x,w_x,z_t,w_t$. \eqref{eqn:Phi-Psi1} and \eqref{eqn:Phi-Psi2} satisfy the Riccati equation 
along the characteristic lines. By this property and the comparison theorem, we will prove that 
\eqref{eqn:Phi-Psi1} and \eqref{eqn:Phi-Psi2} are uniformly bounded.

\section{Invariant region}

In this section, we prove that $\Delta_{m,x}$, $\Delta_{r,x}$ and $\Delta_{l,x}$ are invariant regions for (P1), (P2) and (P3) respectively. This yields the uniform bound of $z$ and $w$.

Let us prove the following.
\begin{proposition}\label{pro:1}
	If 
	\begin{align}
		(z_0(x),w_0(x))\in \Delta_{m,x}\quad\text{ for any $x\geq0$}
	\end{align}
	and (P1) has a smooth solution satisfying $\rho\geq0$, then, $\Delta_{m,x}$ is an invariant region for (P1).
\end{proposition}

\begin{proof}
Choosing $\varepsilon>0$ small enough, we define 
\begin{align*}
	\begin{split}
		\Delta_{m,x,\varepsilon}=&\left\{(z,w);
-\left(L_1-\varepsilon\right)e^{-\int^x_0\bar{a}(y)dy}\leq z\leq -\left(U_1+\varepsilon\right)e^{\int^x_0\bar{a}(y)dy},\right.\\
		&\left.\left(L_2+\varepsilon\right)e^{-\int^x_0\bar{a}(y)dy}\leq w\leq 
\left(U_2-\varepsilon\right)e^{\int^x_0\bar{a}(y)dy},w\geq z\right\}
	\end{split}
\end{align*}
and assume that $(z_0(x),w_0(x))\in \Delta_{m,x,\varepsilon}$. Let us 
prove $\Delta_{m,x}$ is an invariant region for (P1) by a contradiction. 
Then, from Theorem \ref{thm:local} and the finite propagation, it suffices to 
solutions on a compact set $K=[0,R]$ for a positive constant $R$.

If this does not hold, there exist $x_*\in K$ and $t_*>0$ such that the following cases occur. 
\begin{itemize}
\setlength{\parskip}{0.1cm} 
\setlength{\itemsep}{0cm} 
\item[(Case 1)\hspace*{-2ex}] \hspace*{2ex}$\underline{z}(x_*,t_*)=-L_1$,\quad $(z(x,t),w(x,t))\in 
{\rm int}(\Delta_{m,x}),\; x\in{\bf R}_+,\;t<t_*,$\vspace*{0.5ex}
\item[(Case 2)\hspace*{-2ex}] \hspace*{2ex}$z(0,t_*)=\underline{z}(0,t_*)=-L_1$,\quad $(z(x,t),w(x,t))\in {\rm int}(\Delta_{m,x}),\; x\in{\bf R}_+,\linebreak t<t_*,$\vspace*{0.5ex}
\item[(Case 3)\hspace*{-2ex}] \hspace*{2ex}$\bar{z}(x_*,t_*)=-U_1$,\quad $(z(x,t),w(x,t))\in {\rm int}(\Delta_{m,x}),\;  x\in{\bf R}_+,\;t<t_*,$\vspace*{0.5ex}
\item[(Case 4)\hspace*{-2ex}] \hspace*{2ex}${z}(0,t_*)=\bar{z}(0,t_*)=-U_1$,\quad $(z(x,t),w(x,t))\in {\rm int}(\Delta_{m,x}),\;  x\in{\bf R}_+,\linebreak t<t_*,$\vspace*{0.5ex}
\item[(Case 5)\hspace*{-2ex}] \hspace*{2ex}$\underline{w}(x_*,t_*)=L_2$,\quad $(z(x,t),w(x,t))\in {\rm int}(\Delta_{m,x}),\;  x\in{\bf R}_+,\;t<t_*,$\vspace*{0.5ex}
\item[(Case 6)\hspace*{-2ex}] \hspace*{2ex}${w}(0,t_*)=\underline{w}(0,t_*)=L_2$,\quad $(z(x,t),w(x,t))\in {\rm int}(\Delta_{m,x})
,\;  x\in{\bf R}_+,\;t<t_*,$\vspace*{0.5ex}
\item[(Case 7)\hspace*{-2ex}] \hspace*{2ex}$\bar{w}(x_*,t_*)=U_2$,\quad $(z(x,t),w(x,t))\in {\rm int}(\Delta_{m,x}),\;  x\in{\bf R}_+,\;t<t_*,$\vspace*{0.5ex}
\item[(Case 8)\hspace*{-2ex}] \hspace*{2ex}${w}(0,t_*)=\bar{w}(0,t_*)=U_2$,\quad $(z(x,t),w(x,t))\in {\rm int}(\Delta_{m,x})
,\;  x\in{\bf R}_+,\;t<t_*,$
\end{itemize}
where $\underline{z}=e^{\int^x_0\bar{a}(y)dy}{z},\;\bar{z}=e^{-\int^x_0\bar{a}(y)dy}z,\;
\underline{w}=e^{\int^x_0\bar{a}(y)dy}w,\;
\bar{w}=e^{-\int^x_0\bar{a}(y)dy}w$ and ${\rm int}(\Delta_{m,x})$ represents the interior of 
$\Delta_{m,x}$.
\vspace*{1ex}

(Case 1) We set $r=w/z$. Then, if 
$(z,w)\in \Delta_{m,x}$, 
from  $\eqref{eqn:invariant1}_1$ and $\eqref{eqn:invariant3}_2$, we notice that
\begin{align*}
-\sigma_1\leq r\leq0.
\end{align*}
Then, if 
$(z,w)\in \Delta_{m,x}$, we deduce from \eqref{eqn:diagonalization}, \eqref{eqn:f(r)>l} and \eqref{eqn:bar a} 
	\begin{align}
		e^{-\int^x_0\bar{a}(y)dy}\left(\underline{z}_t+\lambda_1\underline{z}_x\right)
=&a(x)\dfrac{\gamma-1}8\left(w^2-z^2\right)
+\bar{a}(x)\left(\dfrac{\gamma+1}4z^2+\dfrac{3-\gamma}4zw\right)\nonumber\\
=&a(x)\dfrac{\gamma-1}8z^2\left(r^2-1\right)
+\bar{a}(x)z^2\left(\dfrac{\gamma+1}4+\dfrac{3-\gamma}4r\right)\nonumber\\
>&-l\bar{a}(x)\dfrac{\gamma-1}8z^2\left|r^2-1\right|
+\bar{a}(x)z^2\left(\dfrac{\gamma+1}4+\dfrac{3-\gamma}4r\right)\nonumber\\
\geq& \dfrac{(\gamma-1)\bar{a}(x)|r^2-1|}{8}z^2\left\{f(r)-l\right\}\nonumber\\
\geq&0.
\label{eqn:maximum1}
\end{align}
In this case, $\underline{z}$ attains the minimum at $(x_*,t_*)$. Then, since $\underline{z}_x=0$ and 
$\underline{z}_t\leq0$, 
we can deduce a contradiction from \eqref{eqn:maximum1}.

\vspace*{1ex}

(Case 2) In this case, we find that $\underline{z}_t(0,t_*)\leq0$ and $\underline{z}_x(0,t_*)\geq0$. Observing Remark \ref{rem:1} (2), since $\lambda_1<0$, we can deduce a contradiction from \eqref{eqn:maximum1}.

\vspace*{1ex}

(Case 3) 	We set  $r=w/z$. Then, if 
	$(z,w)\in \Delta_{m,x}$, from $\eqref{eqn:invariant3}_2$, we notice that
		$-\sigma_1\leq r\leq0$.
	Then, if 
	$(z,w)\in \Delta_{m,x}$, we deduce from \eqref{eqn:diagonalization}, \eqref{eqn:f(r)>l} and \eqref{eqn:bar a}  
	\begin{align}
		e^{\int^x_0\bar{a}(y)dy}\left(\bar{z}_t+\lambda_1\bar{z}_x\right)
		=&a(x)\dfrac{\gamma-1}8\left(w^2-z^2\right)
		-\bar{a}(x)\left(\dfrac{\gamma+1}4z^2+\dfrac{3-\gamma}4zw\right)\nonumber\\
		<& -\dfrac{(\gamma-1)\bar{a}(x)|r^2-1|}{8}z^2\left\{f(r)-l\right\}\nonumber\\
		\leq&0.
\label{eqn:maximum2}
	\end{align}
In this case, from the finite propagation, $\bar{z}$ attains the maximum at $(x_*,t_*)$. Then, 
since $\bar{z}_x=0$ and $\bar{z}_t\geq0$, 
we can deduce a contradiction from \eqref{eqn:maximum2}.

\vspace*{1ex}

(Case 4) In this case, we find that $\bar{z}_t(0,t_*)\geq0$ and $\bar{z}_x(0,t_*)\leq0$. Observing Remark \ref{rem:1} (2), since $\lambda_1<0$, we can deduce a contradiction from \eqref{eqn:maximum2}.

\vspace*{1ex}

(Case 5) 		
We set $r=z/w$. Then, if 
$(z,w)\in \Delta_{m,x}$, from $\eqref{eqn:invariant3}_1$, we notice that $-\sigma_1\leq r\leq0$. Then, if 
$(z,w)\in \Delta_{m,x}$, 
we deduce from \eqref{eqn:diagonalization}, \eqref{eqn:f(r)>l} and \eqref{eqn:bar a} 
\begin{align}
	e^{-\int^x_0\bar{a}(y)dy}\left(\underline{w}_t+\lambda_2\underline{w}_x\right)
	=&-a(x)\dfrac{\gamma-1}8\left(w^2-z^2\right)
	+\bar{a}(x)\left(\dfrac{\gamma+1}4w^2+\dfrac{3-\gamma}4zw\right)\nonumber\\
	>& \dfrac{(\gamma-1)\bar{a}(x)|r^2-1|}{8}w^2\left\{f(r)-l\right\}\nonumber\\
	\geq&0.
\label{eqn:maximum3}
\end{align}	
In this case, $\underline{w}$ attains the minimum at $(x_*,t_*)$. Then, since $\underline{w}_x=0$ and $\underline{w}_t\leq0$, 
we can deduce a contradiction from \eqref{eqn:maximum3}.

\vspace*{1ex}

(Case 6) In this case, from the boundary condition $v=0\;(\text{i.e., } w+z=0)$, (Case 4) and $\eqref{eqn:invariant7}_2$, we find that ${w}(0,t_*)=
-{z}(0,t_*)>U_1\geq L_2$.

\vspace*{1ex}

(Case 7) We set $r=z/w$. Then, if 
$(z,w)\in \Delta_{m,x}$, from $\eqref{eqn:invariant1}_2$ and $\eqref{eqn:invariant3}_1$, we 
notice that $-\sigma_1\leq r\leq0$. Then, if 
$(z,w)\in \Delta_{m,x}$, 
we deduce from \eqref{eqn:diagonalization}, \eqref{eqn:f(r)>l} and \eqref{eqn:bar a}  
	  \begin{align}
	e^{\int^x_0\bar{a}(y)dy}\left(\bar{w}_t+\lambda_2\bar{w}_x\right)
	=&-a(x)\dfrac{\gamma-1}8\left(w^2-z^2\right)
	-\bar{a}(x)\left(\dfrac{\gamma+1}4w^2+\dfrac{3-\gamma}4zw\right)\nonumber\\
	<& -\dfrac{(\gamma-1)\bar{a}(x)|r^2-1|}{8}w^2\left\{f(r)-l\right\}\nonumber\\
	\leq&0.
\label{eqn:maximum4}
\end{align}In this case, $\bar{w}$ attains the maximum at $(x_*,t_*)$. Then, since $\bar{w}_x=0$ and $\bar{w}_t\geq0$, 
we can deduce a contradiction from \eqref{eqn:maximum4}.

\vspace*{1ex}

(Case 8) In this case, from the boundary condition $v=0\;(\text{i.e., } w+z=0)$, (Case 3) and $\eqref{eqn:invariant7}_1$, we find that ${w}(0,t_*)=
-{z}(0,t_*)<L_1\leq U_2$.

Since $\varepsilon$ is arbitrary, we can complete the proof of 
Proposition \ref{pro:1}.
	\end{proof}

We will prove the following.
\begin{proposition}\label{pro:2}
	If 
	\begin{align}
		(z_0(x),w_0(x))\in \Delta_{r,x}\quad (z_B(t),w_B(t))\in \Delta_{r,0}\quad
		\text{ for any $x\geq0$ and $t\geq0$}
	\end{align}
	and (P2) has a smooth solution satisfying $\rho\geq0$, $\Delta_{r,x}$ is an invariant region for  (P2).
\end{proposition}

\begin{proof}
	We first set $\underline{z}=e^{\int^x_0\bar{a}(y)dy}z$ and $r=w/z$. Then, if 
	$(z,w)\in \Delta_{r,x}$, from $\eqref{eqn:invariant6}$, we notice that
	\begin{align*}
		1< r\leq\sigma_2.
	\end{align*}
	Then, if 
	$(z,w)\in \Delta_{r,x}$, we deduce from \eqref{eqn:diagonalization}, \eqref{eqn:f(r)>l} and \eqref{eqn:bar a}  
	\begin{align*}
		e^{-\int^x_0\bar{a}(y)dy}\left(\underline{z}_t+\lambda_1\underline{z}_x\right)
		=&a(x)\dfrac{\gamma-1}8\left(w^2-z^2\right)
		+\bar{a}(x)\left(\dfrac{\gamma+1}4z^2+\dfrac{3-\gamma}4zw\right)\nonumber\\
		>& \dfrac{(\gamma-1)\bar{a}(x)|r^2-1|}{8}z^2\left\{f(r)-l\right\}\nonumber\\
		\geq&0.
	\end{align*}
	We next set $\bar{z}=e^{-\int^x_0\bar{a}(y)dy}z$ and $r=w/z$. Then, if 
	$(z,w)\in \Delta_{r,x}$, from $\eqref{eqn:invariant4}_1$ and $\eqref{eqn:invariant6}$, we notice that
	\begin{align*}
		1< r\leq\sigma_2.
	\end{align*}
	Then, if 
	$(z,w)\in \Delta_{r,x}$, we deduce from \eqref{eqn:diagonalization}, \eqref{eqn:f(r)>l} and \eqref{eqn:bar a}  
	
	\begin{align*}
		e^{\int^x_0\bar{a}(y)dy}\left(\bar{z}_t+\lambda_1\bar{z}_x\right)
		=&a(x)\dfrac{\gamma-1}8\left(w^2-z^2\right)
		-\bar{a}(x)\left(\dfrac{\gamma+1}4z^2+\dfrac{3-\gamma}4zw\right)\nonumber\\
		\leq& -\dfrac{(\gamma-1)\bar{a}(x)|r^2-1|}{8}z^2\left\{f(r)-l\right\}\nonumber\\
		\leq&0.
	\end{align*}

	We next set $\underline{w}=e^{\int^x_0\bar{a}(y)dy}w$ and $r=z/w$. Then, if 
	$(z,w)\in \Delta_{r,x}$, from $\eqref{eqn:invariant5}$, we notice that $0\leq r<1$. Then, if 
	$(z,w)\in \Delta_{r,x}$, 
	we deduce from \eqref{eqn:diagonalization}, \eqref{eqn:f(r)>l} and \eqref{eqn:bar a}  
	  
	\begin{align*}
		e^{-\int^x_0\bar{a}(y)dy}\left(\underline{w}_t+\lambda_2\underline{w}_x\right)
		=&-a(x)\dfrac{\gamma-1}8\left(w^2-z^2\right)
		+\bar{a}(x)\left(\dfrac{\gamma+1}4w^2+\dfrac{3-\gamma}4zw\right)\nonumber\\
		>& \dfrac{(\gamma-1)\bar{a}(x)|r^2-1|}{8}w^2\left\{f(r)-l\right\}\nonumber\\
		\geq&0.
	\end{align*}

	Similarly, we first set $\bar{w}=e^{-\int^x_0\bar{a}(y)dy}w$ and set $r=z/w$. Then, if 
	$(z,w)\in \Delta_{r,x}$, from $\eqref{eqn:invariant4}_2$ and $\eqref{eqn:invariant5}$, we 
	notice that $0\leq r<1$. Then, if 
	$(z,w)\in \Delta_{r,x}$, 
	we deduce from \eqref{eqn:diagonalization}, \eqref{eqn:f(r)>l} and \eqref{eqn:bar a}  
	 
	\begin{align}
		e^{\int^x_0\bar{a}(y)dy}\left(\bar{w}_t+\lambda_2\bar{w}_x\right)
		=&-a(x)\dfrac{\gamma-1}8\left(w^2-z^2\right)
		-\bar{a}(x)\left(\dfrac{\gamma+1}4w^2+\dfrac{3-\gamma}4zw\right)\nonumber\\
		<& -\dfrac{(\gamma-1)\bar{a}(x)|r^2-1|}{8}w^2\left\{f(r)-l\right\}\nonumber\\
		\leq&0.
	\end{align}

	Applying the maximum principle to $\underline{z},\bar{z},
\underline{w}$ and $\bar{w}$, we can complete the proof.
\end{proof}
	
We can similarly prove the following.
\begin{proposition}\label{pro:3}
	If 
	\begin{align}
		(z_0(x),w_0(x))\in \Delta_{l,x}\quad\text{ for any $x\geq0$}
	\end{align}
	and (P3) has a smooth solution satisfying $\rho\geq0$, then, $\Delta_{l,x}$ is an invariant region for (P3).
\end{proposition}

\section{Uniform bound of $z_x$ and $w_x$}
In this section, we drive the uniform bound of $z_x$ and $w_x$. To 
do this, we investigate \eqref{eqn:Phi-Psi1} and \eqref{eqn:Phi-Psi2}, 
which introduce in \cite{CHL}.
In view of \cite[Section 2.2]{CHL}, \eqref{eqn:Phi-Psi1} and \eqref{eqn:Phi-Psi2} satisfy 
\begin{align}
\Phi_t+\lambda_1\Phi_x=\mathcal{A}\Phi^2+\mathcal{B}\Phi+\mathcal{C},\label{eqn:Phi}\\
\Psi_t+\lambda_2\Psi_x=\hat{\mathcal{A}}\Psi^2+\hat{\mathcal{B}}\Phi+\hat{\mathcal{C}},\label{eqn:Psi}
\end{align}
where

for $\beta\ne-1\;(\text{i.e., }\gamma\ne5/3)$,
\begin{align*}
&\mathcal{A}(x,t,\beta)=-\dfrac{\beta-1}{2\beta-1}(w-z)^{-\beta},\\
&\mathcal{B}(x,t,\beta)=\dfrac{a(x)}{2\beta(\beta+1)(2\beta-1)}\left\{\beta(\beta^2+3\beta-2)w+(\beta^3+2\beta^2+3\beta-2)z
\right\},\\
&\mathcal{C}(x,t,\beta)=(w-z)^{\beta}\mathcal{C}_1(x,t,\beta),\\
&\hat{\mathcal{A}}(x,t,\beta)=-\dfrac{\beta-1}{2\beta-1}(w-z)^{-\beta},\\
&\hat{\mathcal{B}}(x,t,\beta)=\dfrac{a(x)}{2\beta(\beta+1)(2\beta-1)}\left\{\beta(\beta^2+3\beta-2)z+(\beta^3+2\beta^2+3\beta-2)w
\right\},\\&\hat{\mathcal{C}}(x,t,\beta)=(w-z)^{\beta}\hat{\mathcal{C}}_1(x,t,\beta)
\end{align*}and 
\begin{align*}\begin{alignedat}{2}
	\mathcal{C}_1(x,t,\beta)=&-\frac{(a(x))^2}{8\beta^2(\beta+1)^2(2\beta-1)}
	\left\{\beta(1-\beta)^2w^2+2\beta(\beta^2+3\beta-2)wz
\right.\\&	\left.
    +(\beta^3+2\beta^2+3\beta-2)z^2\right\}\\
	&-\frac{a_x(x)}{4\beta(\beta+1)(2\beta-1)}\left\{\beta(1-\beta)w^2-2\beta^2wz
	+(2-3\beta-\beta^2)z^2\right\},\\
		\hat{\mathcal{C}}_1(x,t,\beta)=&-\frac{(a(x))^2}{8\beta^2(\beta+1)^2(2\beta-1)}
	\left\{\beta(1-\beta)^2z^2+2\beta(\beta^2+3\beta-2)wz
\right.\\&	\left.
	+(\beta^3+2\beta^2+3\beta-2)w^2\right\}\\
	&-\frac{a_x(x)}{4\beta(\beta+1)(2\beta-1)}\left\{\beta(1-\beta)z^2-2\beta^2wz
	+(2-3\beta-\beta^2)w^2\right\};
\end{alignedat}
\end{align*}

for $\beta=-1\;(\text{i.e., }\gamma=5/3)$,
\begin{align*}
&\mathcal{A}(x,t,-1)=-\dfrac23(w-z),\\
&\mathcal{B}(x,t,-1)=\dfrac{a(x)}{6}\left\{w-4z+4(w-z)\log(w-z)
\right\},\\
&\mathcal{C}(x,t,-1)=(w-z)^{-1}\mathcal{C}_1(x,t,-1),\\
&\hat{\mathcal{A}}(x,t,-1)=-\dfrac23(w-z),\\
&\hat{\mathcal{B}}(x,t,-1)=\dfrac{a(x)}{6}\left\{z-4w-4(w-z)\log(w-z)
\right\},\\
&\hat{\mathcal{C}}(x,t,-1)=(w-z)^{-1}\hat{\mathcal{C}}_1(x,t,-1)
\end{align*}and \textcolor{black}{
\begin{align*}\begin{alignedat}{2}
		\mathcal{C}_1(x,t,-1)=&-\frac{(a(x))^2}{24}
		\bigl\{3w^2+3z^2+2\left(w^2-5wz+4z^2\right)\log(w-z)\\
		&+4(w-z)^2\left(\log(w-z)\right)^2\bigr\}\\
		&+\frac{a_x(x)}{12}\left\{w^2-2wz-5z^2+2\left(w^2+wz-2z^2\right)\log(w-z)\right\},\\
		\hat{\mathcal{C}}_1(x,t,-1)=&-\frac{(a(x))^2}{24}
		\bigl\{3w^2+3z^2+2\left(z^2-5wz+4w^2\right)\log(w-z)\\
		&+4(w-z)^2\left(\log(w-z)\right)^2\bigr\}\\
		&+\frac{a_x(x)}{12}\left\{z^2-2wz-5w^2+2\left(z^2+wz-2w^2\right)\log(w-z)\right\}.
	\end{alignedat}
\end{align*}}

First, for solutions in $\Delta_{m,x}$, we will prove that $z_x$ and $w_x$ are uniformly bounded. 
Let $x(t)$ be the first characteristic line, $\dfrac{dx(t)}{dt}=\lambda_1$ with $x(0)=x_0\geq0$.
We consider \eqref{eqn:Phi} along this line. From \eqref{eqn:Phi}, $\Phi$ satisfies the 
Riccati equation
\begin{align}
	\dfrac{d\Phi(t)}{dt}=\mathcal{A}(t)(\Phi(t))^2+\mathcal{B}(t)\Phi(t)+\mathcal{C}(t),
	\label{eqn:Phi2}
\end{align}where we abbreviate $\Phi(x(t),t),\;\mathcal{A}(x(t),t),\;\mathcal{B}(x(t),t),\;\mathcal{C}(x(t),t)$ as $\Phi(t),\;\mathcal{A}(t),\;\mathcal{B}(t),$\linebreak$\mathcal{C}(t)$, respectively.

When $x(t)\geq0$, we will prove that $\Phi_-(t)=-\delta_1\left(1+M x(t)\right)^{-1-\alpha}$ is a subsolution for \eqref{eqn:Phi2} by choosing 
a positive $M$ large enough, where $\alpha$ and $M$ are defined in \eqref{eqn:H1}. Since our solution contained in $\Delta_{m,x}$, it follows from Remark \ref{rem:1} (6) that 
\begin{align}
	|z(x(t),t)|\leq C_1,\;|w(x(t),t)|\leq C_2,\;w(x(t),t)-z(x(t),t)\geq C_3
	\label{eqn:bound}
\end{align}for some positive constants $C_1,C_2,C_3$. 
In view of \eqref{eqn:condition-initial}, we find that $\Phi_-(0)\leq \Phi(0)\leq\delta_2$.

We then have
\begin{align*}
\left|\mathcal{A}(t)(\Phi_-(t))^2\right|=O(1)\left(1+Mx(t)\right)^{-2-2\alpha},
\end{align*}
where we denote quantities whose moduli satisfy a uniform bound depending only on $C_1,C_2,C_3,\delta_1,\delta_2,k_1,k_2$ by 
$O(1)$.

From \eqref{eqn:H1}, it follows that 
\begin{align*}
\left|\mathcal{B}(t)\Phi_-(t)\right|=O(1)\left(1+Mx(t)\right)^{-2-\frac{3\alpha}{2}},\quad
\left|\mathcal{C}(t)\right|=O(1)(1+Mx(t))^{-2-\alpha}.
\end{align*}
Moreover, from Remark \ref{rem:1} (2), we have
\begin{align}
\lambda_1(x(t),t)<-d_1
\label{eqn:lambda1}
\end{align}
for a positive constant $d_1$. 
Then, from \eqref{eqn:lambda1}, we have
	\begin{align*}
	\dfrac{d\Phi_-(t)}{dt}=(1+\alpha)M\lambda_1\delta_1\left(1+Mx(t)\right)^{-2-\alpha}\leq -(1+\alpha)Md_1\delta_1\left(1+Mx(t)\right)^{-2-\alpha}.
	\end{align*}
We thus have
\begin{align*}
	\dfrac{d\Phi_-(t)}{dt}&-\mathcal{A}(t)(\Phi_-(t))^2-\mathcal{B}(t)\Phi_-(t)-\mathcal{C}(t)\\
\leq&-(1+\alpha)Md_1\delta_1\left(1+Mx(t)\right)^{-2-\alpha}+O(1)\left(1+Mx(t)\right)^{-2-\alpha}\\
\leq&0,
\end{align*}
by fixing $L_1,L_2,U_1,U_2,\alpha,k_1,k_2,\delta_1$ and choosing $M$ large enough  (recall Remark \ref{rem:M}).

Therefore, we obtain 
\begin{align}
	\Phi_-(t)\leq\Phi(t)
	\leq\Phi(0)+\int^t_0\left\{\mathcal{C}(s)-\dfrac{(\mathcal{B}(s))^2}{4\mathcal{A}(s)}\right\}ds. 
	\label{eqn:Phi3}
\end{align}
We prove that the right hand side of \eqref{eqn:Phi3} is 
uniformly bounded for solutions in $\Delta_{m,x}$. 
From (H1) and (H2), since $a^2,a'\in L^1({\bf R})$, we deduce from \eqref{eqn:bound} and \eqref{eqn:lambda1}
\begin{align*}
\int^t_0\left\{\mathcal{C}(s)-\dfrac{(\mathcal{B}(s))^2}{4\mathcal{A}(s)}\right\}ds
\leq& O(1)\int^t_0\left(\left\{a(x(s))\right\}^2+|a'(x(s))|\right) ds\\
=&O(1)\int^{x(t)}_{x(0)}\left(\left\{a(s)\right\}^2+|a'(y)|\right)\dfrac{1}{\lambda_1(x(s),s)}dy\\
\leq&O(1)\int^{x(0)}_{x(t)}\left(\left\{a(s)\right\}^2+|a'(y)|\right)\dfrac{1}{d_1}dy\\
\leq&\dfrac{O(1)}{d_1}(\Vert a^2\Vert_1+\Vert a'\Vert_1).
\end{align*}
Observing \eqref{eqn:Phi3} and \eqref{eqn:bound}, we find that $z_x$ is uniformly bounded. From \eqref{eqn:Psi}, we can similarly prove that $w_x$ is uniformly bounded. In addition, we deduce the uniform bound of 
$z_t$ and $w_t$ from \eqref{eqn:diagonalization}.

Finally, we can similarly prove that $z_x$ and $w_x$ are uniformly bounded for solutions in $\Delta_{r,x}$ and $\Delta_{l,x}$.

\end{document}